\documentclass[12pt, reqno]{amsproc}
\usepackage{amsmath,amssymb,amsthm}
\usepackage[english]{babel}
\usepackage[autostyle]{csquotes}
 2

\usepackage{enumitem}

\usepackage[colorlinks=true,linkcolor=blue,citecolor=blue,pdfpagelabels=false]{hyperref}
\newtheorem{thm}{Theorem}[section]
\newtheorem{lem}[thm]{Lemma}

\newtheorem{cor}[thm]{Corollary}

\newcommand{\thmref}[1]{Theorem~\ref{#1}}
\newcommand{\lemref}[1]{Lemma~\ref{#1}}

\theoremstyle{remark}

\newcommand{\nequiv}{\not\equiv}
\renewcommand{\geq}{\geqslant}
\renewcommand{\leq}{\leqslant}

\newcommand{\medmatrix}[4]{\Bigl(\begin{matrix} #1 \!& \!#2 \\[-4pt] #3 \!&\! #4 \end{matrix}\Bigr)}
\begin{document}

\baselineskip=17pt

\title
{On the parity of the Fourier coefficients of hauptmoduln $j_{N}(z)$ and $j_{N}^{+}(z)$ }
\author[M. Kumari and S. K. Singh]{Moni Kumari and Sujeet Kumar Singh}
\address{National Institute of Science Education and Research, HBNI, Bhubaneswar, Via-Jatni, Khurda, Odisha, 752050, India.}

\email{moni.kumari@niser.ac.in}
\email{sujeet.singh@niser.ac.in}
  
\date{\today}


\begin{abstract}
We obtain some interesting results about the parity of the Fourier coefficients of hauptmoduln $j_{N}(z)$ and $j_{N}^{+}(z),$ for some positive integers $N$. We use elementary methods and the techniques of O. Kolberg's proof for the parity of the partition function.
\end{abstract}

\maketitle

\section{Introduction}
The Klein's j-function plays a central role in several areas of number theory, which is defined by:
\begin{equation}\label{jfun}
j(z):=\frac{E_{4}^{3}(z)}{\Delta(z)}= \frac{1}{q}+744+\sum_{n=1}^{\infty}c(n)q^{n},
\end{equation}
where $z \in \mathbb{H}$ and $q:=e^{2\pi iz}$. Note that, in above $\Delta(z) = q \prod_{n = 1}^\infty(1-q^n)^{24}$ is the Ramanujan's Delta function and $E_4(z)$ is the normalized Eisenstein series of weight $4$.

In the theory of modular functions, the study of congruences for Fourier coefficients has a long history. In particular, congruences for the coefficients $c(n)$ of the Klein's $j$-function have been studied by many mathematicians like, O. Kolberg \cite{Kolberg2}, J. Lehner \cite{Lehner}, K. Ono \cite{Ono}, J. P. Serre \cite{Serre}, among many others. In \cite{Claudia}, C. Alfes says that ``Surprisingly little is known about the behavior of the coefficients $c(n)$ modulo a prime''. Regarding the parity of $c(n),$ it is easy to check that $c(n)$ is even whenever $n\nequiv 7 \pmod{8}.$ This motivates to ask the parity of $c(n)$ in the arithmetic progression $n\equiv7 \pmod{8}.$ A large amount of work has recently been done in this direction, by a variety of methods (see, \cite{Claudia, Murty, Ono}). In 2012, Ono and Ramsey \cite{Ono} proved that there are infinitely many $d$  such that 
$c(nd^2)$ is even whenever $n\equiv 7 \pmod{8}.$ Also, by using the mod $p$ analogue of Atkin-Lehner's theorem and the generalized Borcherds product, they proved that for any $n\equiv 7\pmod{8}$, if there exists one odd integer $d$ such that $c(nd^2)$ is odd, then there are infinitely many odd integers $m$ such that $c(nm^2)$ is odd.
Recently, M. R. Murty and R. Thangadurai \cite{Murty}, in the spirit of Kolberg's \cite{Kolberg1} proof of parity of partition function, gave a range in which a suitable $n\equiv 7 \pmod{8}$ can be chosen such that $c(n)$ is odd (respectively, even).  According to \cite{Claudia}, it is expected that the odd coefficients are supported on the ``one half'' of the arithmetic progression $n\equiv 7 \pmod{8}.$

Motivated from the work of Murty and Thangadurai \cite{Murty}, in this paper we obtain the parity of the Fourier coefficients of hauptmoduln $j_N(z)$ and $j_{N}^+(z),$ defined with respect to the congruence subgroup $\Gamma_0(N)$ and the Fricke group $\Gamma_{0}^+(N)$ respectively, for some particular values of $N$. Here the group $\Gamma_{0}^{+}(N)$ is generated by $\Gamma_{0}(N)$ and the Artin-Lehner involution $W_N=\medmatrix {0} {-1} {N} {0}.$
For $N=2,3,4,5,7$ and $13$ the groups $\Gamma_{0}(N)$ and $\Gamma_{0}^+(N)$ have genus zero and the corresponding hauptmoduln $j_{N}(z)$ and $j_{N}^{+}(z)$ (see, \cite{Mats}) can be expressed by using the Dedekind $\eta$-function $\eta(z):=q^{\frac{1}{24}}\prod_{n=1}^{\infty}(1-q^n)$ in the following way.
\begin{equation}\label{lN}
j_{N}(z):= \left(\frac{\eta(z)}{\eta(Nz)} \right)^{\frac{24}{N-1}}+\frac{24}{N-1},
\end{equation}

\begin{equation}\label{jN+}
j_{N}^{+}(z):= \left(\frac{\eta(z)}{\eta(Nz)} \right)^{\frac{24}{N-1}}+\frac{24}{N-1}+   N^{\frac{12}{N-1}}\left(\frac{\eta(Nz)}{\eta(z)} \right)^{\frac{24}{N-1}}.
\end{equation}

For any positive integer $N$, we always denote the $n$-th Fourier coefficients of $j_{N}(z)$ and $j_{N}^{+}(z)$ by $c_{N}(n)$ and 
$c_{N}^{+}(n)$ respectively. From the definitions \eqref{lN} and \eqref{jN+}, it is clear that $c_{N}(n)$ and $c_{N}^{+}(n)$ are integers. Also, note that for 
 $N=2,3,4,5,7,$ and $13$
\begin{equation*}
j_{N}^{+}(z) = j_{N}(z)+N^{\frac{12}{N-1}}\left(\frac{\eta(Nz)}{\eta(z)} \right)^{\frac{24}{N-1}}.
\end{equation*}
The above equation suggests that the method used to study the Fourier coefficients $c_{N}^+(n)$ may also be applicable to get similar results for $c_{N}(n)$. In our setting, regarding the parity, this is exactly true.
Therefore, we study the parity of the coefficients $c_{N}^{+}(n)$ for $N=2,3,4,7$ and give a detail proof of our statements in sections $2$-$4$. By using the same techniques, we have obtained the analogous results for $c_{N}(n)$, which is listed in the table given in section $5$. Note that for $N=5$ and $13$, similar methods are applicable for the study of the parity of $c_N(n)$ and we have included the result in the table.  To the best of our knowledge, this is the first draft which deals with parity of coefficients of hauptmoduln for congruence subgroups as well as for Fricke groups.

\section{Parity of the Fourier coefficients of $j_{2}^{+}(z)$ and $
j_{4}^{+}(z)$} 
In this section, first we prove that the Fourier coefficients of $j(z), j_2^+(z)$ and $j_4^+(z)$ have the same parity. Then we use the earlier results of \cite{Murty} for $j(z)$ function to show the existence of infinitely many even and odd values for $c_{2}^{+}(n)$ (and also for $c_{4}^{+}(n)$) in the arithmetic progression $n\equiv 7\pmod{8}$.

\begin{thm}\label{congruence24} 
Let $c(n), c_2^+(n)$ and $c_4^+(n)$ denotes the $n$-th Fourier coefficients of $j(z), j_2^+(z)$ and $j_4^+(z)$ respectively. Then we have 
\begin{equation}
 c(n)\equiv c_{2}^{+}(n)\equiv c_{4}^{+}(n)\pmod{2}.
 \end{equation}
\end{thm}

\begin{proof}
From the definition of the function $j_2^+(z)$ given in \eqref{jN+}, we have
\begin{equation}\label{j2+}
j_{2}^{+}(z) \equiv \left(\frac{\eta(z)}{\eta(2z)} \right)^{24} \pmod{2}.
\end{equation}
We know that
\begin{equation}\label{j2+id1}
\left(\frac{\eta(z)}{\eta(2z)} \right)^{24}= \frac{1}{q}\prod_{n=1}^{\infty}\frac{(1-q^n)^{24}}{(1-q^{2n})^{24}} = \frac{1}{q}\prod_{n=1}^{\infty}\frac{1}{(1+q^{n})^{24}}.
\end{equation}
Using the binomial theorem and the fact that $(1+q^n)\equiv(1-q^n)\pmod{2}$, an easy calculation yields 
\begin{equation*}
\frac{1}{q}\prod_{n=1}^{\infty}\frac{1}{(1+q^{n})^{24}} \equiv\frac{1}{q}\prod_{n=1}^{\infty}\frac{1}{(1-q^{n})^{24}}
\equiv\frac{1}{\Delta(z)}\pmod{2}.
\end{equation*}
Combining \eqref{j2+} and \eqref{j2+id1} together with the above congruence gives
\begin{equation}\label{j2+mod2}
j_{2}^{+}(z) \equiv \frac{1}{\Delta(z)}\pmod{2}.
\end{equation}
By using similar arguments as above, we show that
%j_{4}^{+}(z)\equiv\left(\frac{\eta(z)}{\eta(4z)}\right)^8 \equiv\frac{1}{q}\prod_{n=1}^{\infty}\frac{(1-q^n)^{8}}{(1-q^{4n})^{8}} \equiv\frac{1}{q}\prod_{n=1}^{\infty}\frac{1}{(1-q^{n})^{24}} \equiv\frac{1}{\Delta(z)}\pmod{2}.
\begin{equation}\label{j4+mod2}
j_{4}^{+}(z)\equiv \frac{1}{\Delta(z)}\pmod{2}.
\end{equation}
It is well-known that $E_{4}^{3}(z)\equiv 1 \pmod{2}$ and using it in \eqref{jfun}, we get
\begin{equation}\label{jfunction}
j(z)\equiv \frac{1}{\Delta(z)}\pmod{2}.
\end{equation}
\noindent
Combining \eqref{j2+mod2}, \eqref{j4+mod2} and \eqref{jfunction}, we get
$$j(z)\equiv j_{2}^{+}(z)\equiv j_{4}^{+}(z) \pmod{2},$$
which completes the proof.
\end{proof}

In the literature there are many results about the parity of the coefficients $c(n)$ of the Klein's $j$-function. From the above theorem the corresponding results holds good for $c_{2}^{+}(n)$ and $c_{4}^{+}(n)$. In particular, using \cite[Theorem 1.1]{Murty} and \cite[Theorem 1.3]{Murty}, a direct application of \thmref{congruence24} gives infinitely many intervals, each of which contains an integer $n\equiv 7\pmod{8}$ such that $c_{2}^{+}(n)$ (and $c_4^+(n)$) is even (respectively, odd). More precisely, we get the following results.

\begin{cor}\label{2.1}
For $N=2,4$ and for every positive integer $t\geq 1,$
\begin{enumerate}[label=(\alph{*})]
\item
the interval $I_{N,t}:=[t,4t(t+1)-1]$ contains an integer $n\equiv 7\pmod{8}$ with $c_{N}^{+}(n)$ odd.
\item
the interval $J_{N,t}:=[16t-1,(4t+1)^2-1]$ contains an integer $n\equiv 7\pmod{8}$ with $c_{N}^{+}(n)$ even.
\end{enumerate}
In particular, $c_{N}^{+}(n)$ takes both even and odd values infinitely often in the arithmetic progression $n\equiv 7\pmod{8}.$
\end{cor}

\section{Parity of the Fourier coefficients of $j_{3}^{+}(z)$}
%In this section we prove some results about the parity of Fourier coefficients of Houptomodul $j_{3}^{+}(z),$ where
From the definition given in \eqref{jN+}, we have
\begin{align}\label{j3+id1}
j_{3}^{+}(z)&= \left(\frac{\eta(z)}{\eta(3z)} \right)^{12}+12+3^{6}\left(\frac{\eta(3z)}{\eta(z)} \right)^{12}
\equiv\left(\frac{\eta(z)}{\eta(3z)} \right)^{12}+\left(\frac{\eta(3z)}{\eta(z)} \right)^{12} \pmod{2}.
\end{align}
In view of the right hand side of the above congruence, we introduce two functions $f_{3}^{+}$ and $g_{3}^{+}$, defined by
\begin{equation}\label{f3+def}
f_{3}^{+}(z)=\left(\frac{\eta(z)}{\eta(3z)}\right)^{12} =\sum\limits_{n \geq -1} a_3^+(n)q^n,
\end{equation}
and
\begin{equation}\label{g3+def}
g_{3}^{+}(z)=\left(\frac{\eta(3z)}{\eta(z)}\right)^{12}=\sum\limits_{n \geq 1} b_3^+(n)q^n.
\end{equation}
Here $a_{3}^{+}(n)$ and $b_{3}^{+}(n)$ are the $n$-th Fourier coefficients of $f_{3}^{+}$ and $g_{3}^{+}$ respectively.
By using the product formula for eta function and simplifying further, we have
\begin{align}\label{f3}
 f_{3}^{+}(z)=\frac{1}{q}\prod_{n=1}^{\infty}\frac{(1-q^n)^{12}}{(1-q^{3n})^{12}}\equiv \sum_{n=-1}^{\infty}a_{3}^{+}(4n+3)q^{4n+3} \pmod{2}.
 \end{align}
Similarly
 \begin{align}\label{g3}
 g_{3}^{+}(z)\equiv \sum_{n=1}^{\infty}b_{3}^{+}(4n+1)q^{4n+1} \pmod{2}.
 \end{align}
We note that $a_{3}^{+}(n)$ and $b_{3}^{+}(n)$ are even whenever $n\nequiv 3\pmod{4}$ and 
$n\nequiv 1\pmod{4}$ respectively. Therefore, from \eqref{j3+id1} we have
$$c_{3}^{+}(n)\equiv a_{3}^{+}(n)+b_{3}^{+}(n)\pmod{2},$$
 for all $n\geq 1.$
More precisely,
\[ c_{3}^{+}(n) \equiv
   \begin{cases}
    a_{3}^{+}(n) \pmod{2} & ~\text{if}~ n\equiv 3\pmod{4},\\
    b_{3}^{+}(n) \pmod{2} & ~\text{if}~ n\equiv 1 \pmod{4},\\
    0 ~~~~~~\pmod{2} &~ \text{otherwise}.\\
   \end{cases} 
     \]
\noindent
For $n\equiv 0,2 \pmod{4},~ c_{3}^{+}(n)$ is even and therefore, it is natural to ask the parity of 
$c_{3}^{+}(n)$ in the arithmetic progressions $n\equiv 1\pmod{4}$ and $n\equiv 3\pmod{4}.$ Furthermore, the above congruence shows that it is sufficient to study the parity of $a_3^+(n)$ and $b_3^+(n)$. The following theorems give the answer.

\begin{thm}\label{a3+theorem}
Let $t$ be a positive integer. 
\begin{enumerate}[label=(\alph{*})]
\item
Assume $3t(t+1)$ is not of the form $l(l+1)$ for any 
$l \in \mathbb{N}.$ Then the interval $I_{3,t}:=[12t-1,6t(t+1)-1]$ contains an integer $n\equiv 3\pmod{4}$, with $a_3^+(n)$ odd. 
  \item 
 Assume $6t(2t+1)$ is not of the form $l(l+1)$ for any 
$l \in \mathbb{N}.$ Then the interval $J_{3,t}:=[24t-1,12t(2t+1)-1]$ contains an integer  $n\equiv 3\pmod{4}$, such that $a_3^+(n)$ is even. 
 \end{enumerate}
 In particular, there are infinitely many integers $n\equiv 3\pmod{4}$ for which $c_3^+(n)$ is even (respectively, odd).
\end{thm}

The next statement is an analogous result of the Fourier coefficients $b_3^+(n).$
\begin{thm}\label{b3+theorem}
Let $t$ be a positive integer. 
\begin{enumerate}[label=(\alph{*})]
\item
If $t(t+1)$ is not of the form $3l(l+1)$ for any $l \in \mathbb{N},$ then
 the interval $I'_{3,t}:=[4t+1,2t(t+1)+1]$ contains an integer $n\equiv 1\pmod{4}$, such that $b_3^+(n)$ is odd. 
 %In particular, there are infinitely many integers $n\equiv 1\pmod{4}$ for which $b_3^+(n)$ is an odd integer.  
 \item
 If $2t(2t+1)$ is not of the form $3l(l+1)$ for any $l \in \mathbb{N},$ then
the interval $J'_{3,t}:=[8t+1,4t(2t+1)+1]$ contains an integer $n\equiv 1\pmod{4}$, such that $b_3^+(n)$ is even. 
%In particular, there are infinitely many integers $n\equiv 1\pmod{4}$ for which $b_3^+(n)$ is an even integer.   
 \end{enumerate}
 In particular, there are infinitely many integers $n\equiv 1\pmod{4}$ for which $c_3^+(n)$ is even (respectively, odd).
\end{thm}

%\section{proof of \thmref{3.1}}
\noindent
To obtain these results, we prove following lemmas related to the parity of the Fourier coefficients $a_3^+(n)$ and $b_3^+(n)$. 
\begin{lem}\label{a3+cong}
For any integer $n,$ we have
\[ \sum_{k=0}^{\infty}a_{3}^{+}(n-6k(k+1)-1)\equiv
   \begin{cases}
   1 \pmod{2} & ~\text{if}~ n=2l(l+1) ~\text{for some}~ l\in \mathbb{N},\\
   0 \pmod{2} &~ \text{otherwise}.\\
   \end{cases}
       \]
\end{lem}

\begin{proof}
We know that $(x+y)^{2^m}\equiv x^{2^m}+y^{2^m} \pmod{2}$, where $x$ and $y$ are integers. Using this, we have
\begin{equation}\label{etaid1}
\eta^{12}(z)=q^{1/2}\prod_{n=1}^{\infty}(1-q^n)^{12}\equiv q^{1/2}\prod_{n=1}^{\infty}(1-q^{4n})^{3} \pmod{2}.
\end{equation}
Now, using the well-known Jacobi identity
\begin{equation*}
\prod_{n=1}^{\infty}(1-q^n)^3= \sum_{k=1}^{\infty}(-1)^k(2k+1)q^{k(k+1)/2},
\end{equation*}
in the last expression of \eqref{etaid1}, we obtain
\begin{equation}\label{etaid2}
\eta^{12}(z)\equiv q^{1/2}\sum_{n=0}^{\infty}q^{2n(n+1)} \pmod{2}.
\end{equation}
Using the similar arguments, we also get
\begin{equation}\label{etaid3}
\eta^{12}(3z)\equiv q^{3/2}\sum_{n=0}^{\infty}q^{6n(n+1)} \pmod{2}.
\end{equation}
By using \eqref{etaid2} and \eqref{etaid3} in the definition of $f_3^+$ given in \eqref{f3+def}, we get
\begin{equation*}
\sum_{n=-1}^{\infty}a_{3}^{+}(n)q^n\equiv \frac{1}{q}\frac{\sum_{n=0}^{\infty}q^{2n(n+1)}}{\sum_{n=0}^{\infty}q^{6n(n+1)}}\pmod{2}.
\end{equation*}
\noindent
Simplifying above, yields
\begin{equation*}
\sum_{n=0}^{\infty}\left(\sum_{k=0}^{\infty}a_{3}^{+}(n-6k(k+1)-1)\right)q^n\equiv
\sum_{n=0}^{\infty}q^{2n(n+1)} \pmod{2}.
\end{equation*}
Comparing the corresponding coefficients of $q^n$ from the both sides of the above congruence, completes the proof.
\end{proof}
 
\begin{lem}\label{5.1}
For any integer $n\geq 1,$ we have 
\[ \sum_{k=0}^{\infty}b_{3}^{+}(n-2k(k+1)+1)\equiv
   \begin{cases}
   1 \pmod{2} & ~\text{if}~ n=6l(l+1) ~\text{for some}~ l\in \mathbb{N},\\
   0 \pmod{2} &~ \text{otherwise}.\\
   \end{cases}
       \]
\end{lem} 
\begin{proof}
This follows by the same method as the proof of \lemref{a3+cong}.
\end{proof}
Now we are ready to prove our theorems.
\subsection*{\textit{Proof of \thmref{a3+theorem}} (a).}
Let $n=6t(t+1)$ and so from the assumption $n$ is not of the form $2l(l+1)$ for any $l$. Substituting this value of $n$ in \lemref{a3+cong}, we get
\begin{equation*}
\sum_{k=0}^{\infty}a_{3}^{+}(6t(t+1)-6k(k+1)-1)\equiv 0 \pmod{2}.
\end{equation*}
Note that the above sum is finite since $a_{3}^{+}(n)=0$ if $n<-1$. Thus the above congruence is actually
\begin{equation*}
\sum_{k=0}^{t-1}a_{3}^{+}(6t(t+1)-6k(k+1)-1)+a_{3}^{+}(-1)\equiv 0 \pmod{2}.
\end{equation*}
Equivalently, by using  $a_{3}^{+}(-1)=1,$ we write
\begin{equation*}
\sum_{k=0}^{t-1}a_{3}^{+}(\alpha_{3,t}(k)) \equiv 1 \pmod{2},
\end{equation*}
where  $\alpha_{3,t}(k):=6t(t+1)-6k(k+1)-1$ which is $\equiv 3\pmod{4}.$ 
Hence there exist some $k$ such that  $0\leq k\leq t-1$ and $a_{3}^{+}(\alpha_{3,t}(k))$ is odd. Now to complete the proof, it is sufficient to show that for each  $0\leq k\leq t-1$, $\alpha_{3,t}(k)\in I_{3,t}$. This follows from the fact that $\alpha_{3,t}(k)$ is decreasing function of $k$ and $\alpha_{3,t}(0), \alpha_{3,t}(t-1)\in I_{3,t}.$
\qed

\subsection*{\textit{Proof of \thmref{a3+theorem}} (b).}
We prove the result by contradiction. Suppose that $a_{3}^{+}(m)$ is odd whenever $m\equiv 3\pmod{4}$ and $m\in J_{3,t}.$
Put $n=12t(2t+1)$ in \lemref{a3+cong} to get
\begin{equation*}
\sum_{k=0}^{\infty}a_{3}^{+}(12t(2t+1)-6k(k+1)-1)\equiv 0 \pmod{2}.
\end{equation*}
Using the similar arguments as before, we get
\begin{equation}\label{4.3}
\sum_{k=0}^{2t-1}a_{3}^{+}(\beta_{3,t}(k))\equiv 1 \pmod{2},
\end{equation}
where $\beta_{3,t}(k):=12t(2t+1)-6k(k+1)-1$ which is $\equiv 3\pmod{4}.$ Again, one can easily see that $\beta_{3,t}(k)\in J_{3,t},$ for all $0\leq k\leq 2t-1.$ 
 
Now, the number of terms in the left hand side of \eqref{4.3} is even, namely $2t$ and by the assumption on $a_{3}^{+}(m)$ each term is odd. Therefore, the left hand side of \eqref{4.3} is $\equiv 0\pmod{2},$ which is a contradiction.
\qed

\subsection*{\textit{Proof of the \thmref{b3+theorem}}.}
Putting $n=2t(t+1$)$~(\mbox{respectively}, =4t(2t+1))$ in the \lemref{5.1} and following the method of the proof of \thmref{a3+theorem} (a) (respectively, \thmref{a3+theorem} (b)), we get the result.
\qed

 \section{Parity of the Fourier coefficients of $j_{7}^{+}(z)$}
In this section we investigate the parity of the coefficients $c_{7}^{+}(n)$. Form the definition \eqref{jN+}, we have
\begin{align}\label{j7+}
j_{7}^{+}(z)&
%= \left(\frac{\eta(z)}{\eta(7z)} \right)^{4}+4+7^{2}\left(\frac{\eta(7z)}{\eta(z)} \right)^{4}
\equiv\left(\frac{\eta(z)}{\eta(7z)} \right)^{4}+\left(\frac{\eta(7z)}{\eta(z)} \right)^{4} \pmod{2}.
\end{align}
To make our calculations easy, we introduce two functions $f_{7}^{+}$ and $g_{7}^{+}$ with Fourier coefficients $a_7^+(n)$ and $b_7^+(n)$ respectively (as in the case of $j_{3}^{+}(z)$), 
\begin{equation}\label{f7g7}
f_{7}^{+}(z):=\left(\frac{\eta(z)}{\eta(7z)}\right)^{4} {\rm ~and~}
g_{7}^{+}(z):=\left(\frac{\eta(7z)}{\eta(z)}\right)^{4}.
\end{equation}
From \eqref{j7+} and \eqref{f7g7}, it is clear that to determine the parity of $c_7^+(n)$, it is sufficient to know the parity of $a_7^+(n)$ and $b_7^+(n)$. More precisely, for all $n\geq 1$, we have
$$c_{7}^{+}(n)\equiv a_{7}^{+}(n)+b_{7}^{+}(n)\pmod{2}.$$ 
 Now simplifying the eta quotients, we get
\begin{equation*}
f_{7}^{+}(z)=\frac{1}{q}\prod_{n=1}^{\infty}\frac{(1-q^n)^{4}}{(1-q^{7n})^{4}}\equiv \sum_{n=-1}^{\infty}a_{7}^{+}(4n+3)q^{4n+3} \pmod{2},
\end{equation*}
and similarly
\begin{equation*}
g_{7}^{+}(z)\equiv \sum_{n=1}^{\infty}b_{7}^{+}(4n+1)q^{4n+1} \pmod{2}.
\end{equation*}
Note that $a_{7}^{+}(n)$ and $b_{7}^{+}(n)$ are even whenever $n\nequiv 3\pmod{4}$ and 
$n\nequiv 1\pmod{4}$ respectively. Combining the above discussions, we get
\[ c_{7}^{+}(n) \equiv
   \begin{cases}
    a_{7}^{+}(n) \pmod{2} & ~\text{if}~ n\equiv 3\pmod{4},\\
    b_{7}^{+}(n) \pmod{2} & ~\text{if}~ n\equiv 1 \pmod{4},\\
    0 ~~~~~~\pmod{2} &~ \text{otherwise}.\\
   \end{cases} 
     \]

For $n\equiv 0,2 \pmod{4},~c_{7}^{+}(n)$ is even. It is natural to investigate the parity of $c_{7}^{+}(n)$ in the arithmetic progression $n\equiv 1\pmod{4}$ and $ n\equiv 3 \pmod{4}.$ The following theorems answer these questions. 

\begin{thm}\label{a7+id}
Let $t$ be a positive integer.
\begin{enumerate}[label=(\alph{*})]
\item
Assume that $7t(3t-1)$ is neither of the form $l(3l-1)$ nor $l(3l+1)$ for any $l \in \mathbb{N}.$ Then
the interval $I_{7,t}:=[56t-29,14t(3t-1)-1]$ contains an integer $n\equiv 3\pmod{4}$, with $a_7^+(n)$ odd. 
\item
Assume that $14t(6t-1)$ is neither of the form $l(3l-1)$ nor $l(3l+1)$ for any $l \in \mathbb{N}.$ Then
the interval $J_{7,t}:=[112t-29,28t(6t-1)-1]$ contains an integer $n\equiv 3\pmod{4}$, with $a_7^+(n)$ even.  
 \end{enumerate}
 In particular, $c_{7}^{+}(n)$ takes both even and odd values infinitely often in the arithmetic progression $n\equiv 3\pmod{4}.$
\end{thm}

\begin{thm}\label{b7+id}
Let $t$ be a positive integer.
\begin{enumerate}[label=(\alph{*})]
\item
Assume that $t(3t-1)$ is neither of the form $7l(3l-1)$ nor $7l(3l+1)$ for any $l \in \mathbb{N}.$ Then
the interval $I'_{7,t}:=[8t-3,2t(3t-1)+1]$ contains an integer $n\equiv 1\pmod{4}$, with $b_7^+(n)$ odd. 
\item
Assume that $2t(3t-1)$ is neither of the form $7l(3l-1)$  nor $7l(3l+1)$ for any $l \in \mathbb{N}.$ Then
the interval $J'_{7,t}:=[16t-3,4t(6t-1)+1]$ contains an integer $n\equiv 1\pmod{4}$, with $b_7^+(n)$ even.    
\end{enumerate}  
 In particular, $c_{7}^{+}(n)$ takes both even and odd values infinitely often in the arithmetic progression $n\equiv 1\pmod{4}.$
\end{thm}

We need the following lemmas to prove \thmref{a7+id} and \thmref{b7+id}. 
\begin{lem}\label{a7+cong}
For any integer $n,$ we have
\begin{align}
&\sum_{k=0}^{\infty}a_{7}^{+}(n-14k(3k-1)-1)+\sum_{k=0}^{\infty}a_{7}^{+}(n-14k(3k+1)-1) \notag\\
& \equiv 
   \begin{cases}
   1 \pmod{2} & ~\text{if}~ n=2l(3l-1)~\text{or}~n=2l(3l+1)~ ~\text{for some}~ l\in \mathbb{N},\\
   0 \pmod{2} &~ \text{otherwise}.\\
   \end{cases}
\end{align}    
\end{lem}

\begin{proof}
\noindent
Using the product identity 
$\displaystyle{\prod_{n=1}^{\infty}(1-q^n)=\sum_{n=-\infty}^{\infty}(-1)^nq^{\frac{n(3n-1)}{2}}}$, we obtain
\begin{equation}\label{a7+id2}
\prod_{n=1}^{\infty}(1-q^{4n})=\sum_{n=-\infty}^{\infty}(-1)^nq^{2n(3n-1)}
\equiv \sum_{n=-\infty}^{\infty}q^{2n(3n-1)} \pmod{2}
\end{equation}
and
\begin{equation}\label{a7+id3}
\prod_{n=1}^{\infty}(1-q^{28n})=\sum_{n=-\infty}^{\infty}(-1)^nq^{14n(3n-1)}
\equiv \sum_{n=-\infty}^{\infty}q^{14n(3n-1)} \pmod{2}.
\end{equation}
Also form \eqref{f7g7}, we have 
\begin{align}\label{a7+id4}
f_{7}^{+}(z)
&=\frac{\eta^{4}(z)}{\eta^{4}(7z)} = \frac{1}{q}\prod_{n=1}^{\infty}\frac{(1-q^n)^4}{(1-q^{7n})^4} \equiv\frac{1}{q}\prod_{n=1}^{\infty}\frac{(1-q^{4n})}{(1-q^{28n})} \pmod{2}.
\end{align}
\noindent
Substituting the congruences  \eqref{a7+id2} and \eqref{a7+id3} in  \eqref{a7+id4}, we get
\begin{equation*}
\sum_{n=-1}^{\infty}a_{7}^{+}(n)q^{n}\equiv  \frac{\sum_{n=0}^{\infty}q^{2n(3n-1)}+\sum_{n=1}^{\infty}q^{2n(3n+1)}}{\sum_{n=0}^{\infty}q^{14n(3n-1)+1}+\sum_{n=1}^{\infty}q^{14n(3n+1)+1}} \pmod{2}.
\end{equation*} 
\noindent
Multiplication of series in above congruence gives
\begin{align*}
& \sum_{n=0}^{\infty}\left(\sum_{k=0}^{\infty}a_{7}^{+}(n-14k(3k-1)-1)+\sum_{k=0}^{\infty}a_{7}^{+}(n-14k(3k+1)-1)\right) q^{n}\\
& \equiv  \left(\sum_{n=0}^{\infty}q^{2n(3n-1)}+\sum_{n=1}^{\infty}q^{2n(3n+1)}\right) \pmod{2}.
\end{align*} 
Now comparing the corresponding coefficients of $q^n$ form the both sides of the above congruence, gives the required result.
\end{proof}
\begin{lem}\label{b7+cong}
For any integer $n\geq 1,$ we have
\begin{align}
&\sum_{k=0}^{\infty}b_{7}^{+}(n-2k(3k-1))+\sum_{k=0}^{\infty}b_{7}^{+}(n-2k(3k+1))\notag\\
&\equiv 
   \begin{cases}
   1 \pmod{2} & ~\text{if}~ n=14l(3l-1)+1~\text{or}~n=14l(3l+1)+1~ ~\text{for some}~ l,\\
   0 \pmod{2} &~ \text{otherwise}.\\
   \end{cases}
\end{align}    
\end{lem}

\begin{proof}
The proof follows by using similar arguments as in the proof of \lemref{a7+cong} and we leave it to the reader.
\end{proof}

\subsection*{\textit{Proof of the \thmref{a7+id}} (a).}
Suppose, contrary to our claim, that $a_{7}^{+}(m)$ is even for all $m\in I_{7,t}.$ Putting $n=14t(3t-1)$ in the \lemref{a7+cong} and using the hypothesis of the theorem, we obtain
\begin{equation*}
\sum_{k=0}^{\infty}a_{7}^{+}\left(14t(3t-1)-14k(3k-1)-1 \right)+ \sum_{k=0}^{\infty} a_{7}^{+}(14t(3t-1)-14k(3k+1)-1)\equiv 0 \pmod{2}.
\end{equation*}
Above congruence can be written as
\begin{equation}\label{a7+sum1}
 \sum_{k=0}^{\infty} a_{7}^{+}\left(\alpha_{7,t}(k)\right) + \sum_{k=0}^{\infty} a_{7}^{+}(\beta_{7,t}(k)) \equiv 0 \pmod{2},
 \end{equation}
where $\alpha_{7,t}(k)=14t(3t-1)-14k(3k-1)-1$ and $\beta_{7,t}(k)=14t(3t-1)-14k(3k+1)-1.$

Since $a_{7}^{+}(m)=0$ for $m<-1$, so the sum on the left hand side of \eqref{a7+sum1} is  finite and the maximum value of $k$ in the sum will be $t$. Therefore, we write the congruence \eqref{a7+sum1} as 
\begin{equation*}
\left(\sum_{k=0}^{t-1} \left\{ a_{7}^{+}(\alpha_{7,t}(k))+a_{7}^{+}(\beta_{7,t}(k))\right\} \right)
+a_{7}^{+}(\alpha_{7,t}(t))+a_{7}^{+}(\beta_{7,t}(t))\equiv 0 \pmod{2}.
\end{equation*}
But $a_{7}^{+}(\alpha_{7,t}^{+}(t))=a_{7}^+(-1)=1$ and $a_{7}^{+}(\beta_{7,t}(t))=a_{7}^{+}(-28t-1)=0$ for any $t\in \mathbb{N}.$
\noindent
Thus, we have 
\begin{equation}\label{a7+rel}
\sum_{k=0}^{t-1}\left\{ a_{7}^{+}(\alpha_{7,t}(k))+a_{7}^{+}(\beta_{7,t}(k))\right\}
\equiv 1\pmod{2}.
\end{equation}
Clearly, for a fix $t$, $\alpha_{7,t}$ and $\beta_{7,t}$ are decreasing functions. In view of  $\alpha_{7,t}(0),\alpha_{7,t}(t-1),\beta_{7,t}(0),\beta_{7,t}(t-1) \in I_{7,t}$, it follows that
$$\alpha_{7,t}(k),~ \beta_{7,t}(k)\in I_{7,t},~~ {\rm for ~ all} ~0\leq k\leq t-1.$$ 
Now the assumption on $a_7^+(m)$ implies that the left hand side of \eqref{a7+rel} is $\equiv 0 \pmod{2},$ which is a contradiction. 
\qed
 
 \subsection*{\textit{Proof of the \thmref{a7+id}} (b).}
 Assume on the contrary that $a_{7}^{+}(m)$ is odd for every $m\in J_{7,t}$ and $m\equiv 3\pmod{4}$.
As before, putting $n=28t(6t-1)$ in the \lemref{a7+cong} and using the hypothesis of the theorem, we have
\begin{equation*}
\sum_{k=0}^{\infty}a_{7}^{+}(28t(6t-1)-14k(3k-1)-1)+\sum_{k=0}^{\infty}a_{7}^{+}(28t(6t-1)-14k(3k+1)-1)\equiv 0 \pmod{2}.
\end{equation*}
%because $n\neq2j(j+1)$ for any $j,$ which is given in the statement.\\
%Notice that this is a finite sum because $a_{7}^{+}(m)=0$ for $m<-1.$ To simply the calculations,
For $k\in \mathbb{N}\cup \{0\},$ we denote,
$\gamma_{7,t}(k)=28t(6t-1)-14k(3k-1)-1$ and $\delta_{7,t}(k)=28t(6t-1)-14k(3k+1)-1.$ 
Using these notations, we write the above congruence, as
\begin{equation*}
\sum_{k=0}^{\infty}a_{7}^{+}(\gamma_{7,t}(k))+\sum_{k=0}^{\infty}a_{7}^{+}(\delta_{7,t}(k))\equiv 0 \pmod{2}.
\end{equation*}
The maximum value of $k$ in the above sum will be $2t$ because after that $a_{7}^{+}(\gamma_{7,t}(k))$ and $a_{7}^{+}(\delta_{7,t}(k))$ are zero.
Therefore the above congruence becomes
\begin{equation}\label{a7+sum}
\left(\sum_{k=0}^{2t-1}\left\{a_{7}^{+}(\gamma_{7,t}(k))+a_{7}^{+}(\delta_{7,t}(k))\right\}\right)
+a_{7}^{+}(\gamma_{7,t}(2t))+a_{7}^{+}(\delta_{7,t}(2t))\equiv 0 \pmod{2}. 
\end{equation}
Note that $a_{7}^{+}(\gamma_{7}^{+}(2t))=a_{7}^{+}(-1)=1$ and $a_{7}^{+}(\delta_{7,t}(2t))=a_{7}^{+}(-56t-1)=0$ for any $t\in \mathbb{N}.$
\noindent
Substituting these values in \eqref{a7+sum}, we get
\begin{equation}\label{a7+rel1}
\sum_{k=0}^{2t-1}\left\{a_{7}^{+}(\gamma_{7,t}(k))+a_{7}^{+}(\delta_{7,t}(k))\right\}
\equiv 1\pmod{2}.
\end{equation}
Arguing as before, it follows that for all $0\leq k\leq 2t-1$,
$\gamma_{7,t}(k), \delta_{7,t}(k)\in J_{7,t}.$

Now, the assumption on $a_{7}^{+}(m)$ and the fact that the sum on the left hand side of the \eqref{a7+rel1} has even number of terms gives a contradiction.
\qed
\subsection*{\textit{Proof of the \thmref{b7+id}}.}
Putting $n=2t(3t-1)+1~(\mbox{respectively},=4t(6t-1)+1)$ in the \lemref{b7+cong} and  following the method of the proof of the \thmref{a7+id} (a) (respectively, \thmref{a7+id} (b)), we get the results.
\qed

\section{Parity of the Fourier coefficients of $j_N(z)$}
In the following table, we give analogues results about the parity of Fourier coefficients of $j_N(z)$ for $N=2,3,4,5,7$ and $13.$ The proof of these statements follow by using the same techniques as we have used in the sections $2,3$ and $4$. We remark that $c_{N}(n)$ is even if $n$ is not in the arithmetic progression mentioned in the second column.

\bigskip

\noindent
\begin{tabular}{ |p{2.5cm}|p{3cm}|p{5cm}|p{5cm}|  }

 %\multicolumn{4}{|c|}{Parity of the Fourier coefficients of $j_{p}(z)$} \\
 \hline
Coefficients of $j_N(z)$ & Arithmetic progressions & interval which contains a positive integer $n$ with $c_{N}(n)$ odd  & interval which contains a positive integer $n$ with $c_{N}(n)$ even\\
 \hline
$c_2(n)$ & $n\equiv7\pmod 8$   & $[t,4t(t+1)-1],$ &   $[16t-1,(4t+1)^2-1].$  \\ 
 & & &\\
 \hline
 $c_3(n)$ & $n\equiv 3\pmod 4$   &  $[12t-1,6t(t+1)-1]$,   & $[24t-1,12t(2t+1)-1]$, \\
% & & &\\
  &   &  $3t(t+1)\neq l(l+1)$ for any $l,$ &  $ \!6t(2t+1)\! \neq \! l(l+1)$ for any $l.$\\
 & & &\\
 \hline
$c_4(n)$ & $n\equiv7\pmod 8$   & $[t,4t(t+1)-1],$ &   $[16t-1,(4t+1)^2-1].$  \\ 
 & & &\\
 \hline
 $c_5(n)$ & $n\equiv 1 \pmod 2$  & $[10t-1,5t(t+1)-1]$, &  $[20t-1,10t(2t+1)-1]$,  \\
% &&&\\
 & &  $5t(t+1)\neq l(l+1)$ for any $l,$  & $\! 10t(2t \! +\! 1)\! \neq \! l(l\! +\! 1)$ for any $l.$ \\
 &&&\\
 \hline
 $c_7(n)$ & $n\equiv 3\pmod 4$ & $[56t-29,14t(3t-1)-1]$, &  $[112t-29,28t(6t-1)-1]$, \\
 % &&&\\
  & &  $7t(3t-1)\neq l(3l-1)\neq l(3l+1)$ for any $l,$ &   $14t(6t-1)\neq l(3l-1)\neq l(3l+1)$ for any $l.$\\
    &&&\\
  \hline
  $c_{13}(n)$ & $n\equiv 1 \pmod 2$ & $[52t-27,13t(3t-1)-1]$, $13t(3t-1)\neq l(3l-1)\neq l(3l+1)$ for any $l,$  & $[104t-27,26t(6t-1)-1]$, $26t(6t-1)\neq l(3l-1)\neq l(3l+1)$ for any $l.$\\
    &&&\\
 \hline 
\end{tabular}
$$\quad $$

\section{Further Remark}
For the proof of our results, we rely on the fact that for $N=2,3,4,5,7$ and $13$ the corresponding hauptmoduln $j_N(z)$ and  $j_{N}^{+}(z)$ have explicit expression in terms of $\eta$-function. Also by using the method of this paper, we could not conclude any result about the parity of the coefficients $c_{N}^{+}(n)$ for $N=5$ and $13$. It would be interesting to extend these results for congruence subgroup and Fricke groups for higher level. Moreover, one may even ask to get some shorter intervals having the same properties as in the statements of this paper.

Similar to the Klein's $j$-function, we expect that the Fourier coefficients of  $j_N(z)$ and  $j_{N}^{+}(z)$ in some arithmetic progression, are both even and odd with density $\frac{1}{2}$.

\bigskip 
 
\noindent{\bf Acknowledgements.} The first author is grateful to her advisor Dr. Brundaban Sahu for his continuous support and making some important corrections in an earlier version of this paper. The authors thank Dr. Jaban Meher for some comments.

\end{document}